\documentclass[12pt]{article}
\usepackage{amsmath, amsthm, amsfonts, mathtools, xcolor, caption, thmtools}

\numberwithin{equation}{section}

\usepackage[style=numeric,sorting=none,giveninits=true]{biblatex}
\usepackage{etoolbox}
\usepackage[margin=1in]{geometry}

\DeclareMathAlphabet{\mymathbb}{U}{BOONDOX-ds}{m}{n}

\usepackage[colorlinks, allcolors=blue, hypertexnames=false]{hyperref}
\usepackage[noabbrev, nameinlink]{cleveref}

\usepackage{comment}

\newtheorem{theorem}{Theorem}[section]
\newtheorem{lemma}[theorem]{Lemma}
\newtheorem{corollary}[theorem]{Corollary}
\newtheorem{proposition}[theorem]{Proposition}
\theoremstyle{definition}
\newtheorem{remark}{Remark}
\newtheorem{definition}{Definition}
\newtheorem{example}{Example}

\newcommand{\ev}[1]{( #1 )} 
\newcommand{\ew}[1]{\left( #1 \right)}

\newcommand{\C}{\mathbb{C}}
\newcommand{\Q}{\mathbb{Q}}

\begin{document}

\title{A general switching method for constructing E-cospectral hypergraphs}

\author{
Aida Abiad
\thanks{\texttt{a.abiad.monge@tue.nl}, Department of Mathematics and Computer Science, Eindhoven University of Technology, 5600 MB Eindhoven, The Netherlands}
\thanks{Department of Mathematics and Data Science, Vrije Universiteit Brussel, 1050 Brussels, Belgium} 
\and
Joshua Cooper \thanks{\texttt{cooper@math.sc.edu}, Department of Mathematics, University of South Carolina, Columbia, SC 29208, U.S.A.} 
\and
Utku Okur
\thanks{\textit{Corresponding author.} \texttt{utku.okur@math.tu-freiberg.de}, Institute of Discrete Mathematics and Algebra, Faculty of Mathematics and Computer Science, Technische Universit\"at Bergakademie Freiberg, 09596 Freiberg, Germany}
}

\date{}
\maketitle

\begin{abstract}
Spectral hypergraph theory studies the structural properties of a hypergraph that can be inferred from the eigenvalues and the eigenvectors of either matrices or tensors associated with it. In this paper we study the spectral indistinguishability in the hypergraph setting.  We present a general switching method to construct uniform $E$-cospectral hypergraphs (hypergraphs with the same $E$-spectrum), and discuss some of its multiple applications. Our method not only provides a framework to unify the existing methods for obtaining $E$-cospectral hypergraphs via switching, but also generalizes most of the existing switching tools, yielding multiple new constructions. Finally, we compare common methods of computing $E$-characteristic polynomials, and in particular show that one standard method, while useful for generic tensors, is uninformative for almost all hypergraphs.

\medskip


\end{abstract}


\section{Introduction}

Spectral hypergraph theory seeks to deduce structural properties about the hypergraph using its spectra. This field has received a lot of attention in the last two decades, see for example \cite{C2012,CD2015,FL1996,LZ2017,R2002,SSP2022,shao,SQH2015,W2022,XW2019,ZKSB2017}. Several directions for future research on higher-order networks are proposed in \cite{Abiad_2026}, including the construction of cospectral hypergraphs. The study of cospectral hypergraphs is important since it reveals which hypergraph properties cannot be deduced from their spectra. While the construction of cospectral graphs has
been investigated extensively in the literature (see e.g. \cite{ah2012, gm1982,HH1988,QJW2020,s1973}), much less is known about the construction of cospectral hypergraphs.

Two $k$-uniform hypergraphs are said to be \emph{cospectral} (\emph{$E$-cospectral}), if their adjacency tensors have the same characteristic polynomial ($E$-characteristic polynomial). A $k$-uniform hypergraph $H$ is said to be \emph{determined by its spectrum}, if there is no non-isomorphic $k$-uniform hypergraph cospectral with $H$. A way to show that two hypergraphs are not determined by their spectra is to construct cospectral pairs. One way to do so is using switching methods, which are a local hypergraph operation that results in \emph{$E$-cospectral hypergraphs} (hypergraphs with the same $E$-spectrum). In this direction, Bu, Zhou, and Wei \cite{bzw2014} presented a switching method for constructing $E$-cospectral hypergraphs which is based on the celebrated Godsil-McKay switching (GM-switching) for graphs, and they also showed that some basic hypergraph classes are determined by their spectra. Abiad and Khramova \cite{AK2024} extended the more recent graph switching method by Wang, Qiu and Hu (WQH-switching) to the hypergraph setting.

In this paper, we provide a general method for constructing $E$-cospectral uniform hypergraphs via switching. This general framework, which uses the language of tensors, subsumes and unifies the known switching methods for obtaining $E$-cospectral hypergraphs \cite{bzw2014,AK2024}, and also yields multiple new switching methods, which in turn generalize multiple results for graphs \cite{abiad_haemers,mao_wang_liu_qiu_2023,ABS2025}. Along the way, we also investigate the relationship between differing definitions of $E$-eigenvalues appearing the literature, especially regarding \cite{Q2007} and \cite{CARTWRIGHT2013942}. In fact, the situation is a bit complicated: Cartwright and Sturmfels define {\em normalized $E$-eigenvalues} to be the same as what Qi defines to be $E$-eigenvalues; the former give a method to compute (normalized) $E$-eigenvalues via elimination ideals, whereas the latter gives a similar method (via resultants) that he shows gives the correct values if the tensor is what he terms ``regular''; however, we show below that for $k\geq 3$, the adjacency hypermatrix of a $k$-uniform hypergraph is regular if and only if the hypergraph is a disjoint union of complete hypergraphs with at most one isolated vertex for $k \geq 3$, but almost all $2$-graphs are indeed regular.

\section{Preliminaries}


Throughout, we shall mostly follow the notation used in \cite{bzw2014}.

For a positive integer $n$, let $[n] = \{1, \ldots , n\}$. For disjoint sets $A_1$ and $A_2$, we write $A_1 \sqcup A_2$ for the disjoint union. We write $\mymathbb{0}_n$, $\mymathbb{1}_n$, $I_n$ and $J_n$ for the all-zeroes vector, all-ones vector, the identity matrix and the all-ones matrix of dimension $n$, respectively. We omit the subscript when the dimension is clear from context. Given two square matrices $A,B$, we write $A\oplus B$ for the square matrix $\begin{bmatrix}
A & 0 \\ 0 & B
\end{bmatrix}$. For a square matrix $A$, the \textit{permanent} of $A$ is denoted by $\textup{per}\ev{A}$. Whenever it is clear from context, we omit curly brackets of a set, e.g., an edge of a hypergraph. 

An order $k$ dimension $n$ \emph{tensor} (also variously called a \emph{hypermatrix} or \emph{multidimensional array}) $\mathcal{A}=(a_{i_1, \ldots, i_k}) \in \mathbb{C}^{n\times \cdots \times n}$ is a multidimensional array with $n^k$ entries, where $i_j \in [n]$, $j=1,\ldots, k$. For example, in case $k=1$, $\mathcal{A}$ is a column vector of dimension $n$, and in case $k=2$, $\mathcal{A}$ is an $n\times n$ square matrix.

A $k$-\emph{uniform hypergraph} is a pair $\ev{V,E}$, where $V$ is a finite set, called the \textit{vertex set}, and $E \subseteq \binom{V}{k}$ is a set of $k$-subsets of $V$, called the \textit{edge set}. Given a subset $C\subseteq V$, then the hypergraph \textit{induced by $C$} is $\ev{C,\{e\in E\ev{G}: e\subseteq C\}}$.



We say that a tensor $\mathcal{A}$ is \emph{symmetric} if $a_{i_1, \ldots, i_k} = a_{i_{\sigma{(1)}}\,i_{\sigma{(2)}}\cdots i_{\sigma{(k)}}}$ for any permutation $\sigma$ on $[k]$. We denote by $\mathcal{I}^k_n$ the \emph{identity tensor} of order $k$ and dimension $n$, \emph{i.e.} a tensor of elements $a_{i_1, \ldots, i_k}$ such that
$$a_{i_1, \ldots i_k}=\begin{cases}
1 & \text{if } i_1=i_2=\ldots =i_k,\\
0 & \text{otherwise}.
\end{cases}$$

The \emph{adjacency tensor} of $G$, denoted by $\mathcal{A}_G$, is an order $k$ dimension $|V(G)|$ tensor with entries \cite{C2012}:
$$ a_{i_1\, \ldots, i_{k}}=\begin{cases}
1 & \text{if } \{i_1, \ldots, i_k \}\in E(G), \\
0 & \text{otherwise.}
\end{cases}$$
Note that the order $k$ of the adjacency tensor corresponds to the rank $k$ of the hypergraph.

\begin{remark}\leavevmode
\label{rmk:symmetry_normalizing} 
\begin{enumerate}
\item[\textit{i)}] From the definition of the adjacency tensor, it is easy to observe that for a hypergraph $G$, the tensor $\mathcal{A}_G$ is symmetric.
\item[\textit{ii)}] The adjacency tensor is sometimes defined with a scaling factor $\frac1{\ev{k-1}!}$ in each entry. For the equations we are concerned with, this scalar cancels out, so the theory applies regardless of whether this factor is included in the definition of $\mathcal{A}_G$. 
\end{enumerate}
\end{remark}

The \textit{generic tensor} $\mathbb{A}^n_k$ of dimension of dimension $n$ and order $k\geq 2$ is the tensor such that the $j_1,\ldots,j_k$ entry is the variable $\mathbf{a}\ev{j_1,\ldots,j_k} $. The entries of the generic tensor belong to the ring $\C[\{\mathbf{a}\ev{j_1,\ldots,j_k}\}]$. 

The following tensor multiplication was introduced by Shao (\cite[Definition 1.1, p.~2352]{shao}) as a generalization of the matrix multiplication.

\begin{definition}\cite{shao}\label{def:tensorproduct}
Let $\mathcal{A}$ and $\mathcal{B}$ be order $m\geq 2$ and order $k\geq 1$, dimension $n$ tensors, respectively. The product $\mathcal{A} \mathcal{B}$ is the following tensor $\mathcal{C}$ of order $(m-1)(k-1)+1$ and dimension $n$ with entries:
$$c_{i_1\,\alpha_1, \ldots, \alpha_{m-1}} = \sum_{i_2,\ldots, i_m\in [n]} a_{i_1,\ldots, i_{m}} b_{i_2\,\alpha_{1}} \cdots b_{i_m\,\alpha_{m-1}},$$
where $i_1\in [n]$, $\alpha_1,\ldots,\alpha_{m-1} \in [n]^{k-1}$.
\end{definition}

In particular, according to \cite[Example 1.1]{shao}, for an order $k\geq 2$ dimension $n$ tensor $\mathcal{A}$ and a vector $x = (x_1,\ldots, x_n)^\top$ we can derive that the product $\mathcal{A} x$ is a vector with $i$-th component calculated by
\begin{equation}
\label{eqn:tensor_times_vect}
(\mathcal{A} x)_i=\sum_{i_2,\ldots, i_k\in [n]} a_{i\, i_2, \ldots, i_{k}} x_{i_2} \cdots x_{i_k}.
\end{equation}

In 2005, Qi \cite{Q2005} and Lim \cite{L2005} independently introduced the concept of tensor eigenvalues with two different definitions. Both of them generalize the notion of matrix eigenvalue in their own way. Since then, a vast number of authors have used such definitions to study spectral properties of hypergraphs \cite{C2012,CD2015,shao,SQH2015,W2022,XW2019,ZKSB2017}. The present manuscript is concerned with the $E$-eigenvalue equations introduced by \cite{Q2005}, which we define next. 

Let $\mathcal{A}$ be an order $k$ dimension $n$ tensor. A number $\lambda \in \mathbb{C}$ is called a \emph{$E$-eigenvalue} of $\mathcal{A}$ if there exists a nonzero vector $x \in \mathbb{C}^n$ such that the system of equations $\mathcal{A}x - \lambda x = 0$ is satisfied. Note that given an eigenpair $\ev{ \lambda, x}$, and a non-zero constant $c\in \C\setminus \{0\}$, then $\ev{ c^{k-2}\lambda, cx}$ is also an eigenpair, in which case, we say that the eigenpairs $\ev{ \lambda, x}$ and $\ev{ c^{k-2}\lambda, cx}$ are \textit{equivalent}. The following normalized system of $n+1$ equations picks out a single representative from each equivalence class:
\begin{equation}
\label{eqn:e_char_res} 
\begin{split}
\mathcal{A}x - \lambda x = 0, \\ 
x^\top x - 1 = 0.
\end{split}
\end{equation}
A complex number $\lambda$ is a \textit{normalized $E$-eigenvalue} of $\mathcal{A}$, if it satisfies the system (\ref{eqn:e_char_res}). For a hypergraph $G$ of rank $k\geq 2$, the $E$-eigenvalues of $G$ are defined as the $E$-eigenvalues of the scaled adjacency tensor $\frac{1}{\ev{k-1}!} \mathcal{A}_G$. 


We consider alternative definitions of the E-characteristic polynomial of a tensor in \Cref{sec:regular_implies_complete}. Sagemath code for the calculation of the E-characteristic polynomial of small examples can be found in \cite{github}. 

We say that two tensors $\mathcal{A},\mathcal{A}'$ are \emph{$E$-cospectral} if they have the same set of normalized $E$-eigenvalues. Furthermore, two hypergraphs $G,H$ are \emph{$E$-cospectral}, if their adjacency tensor are E-cospectral. We now introduce the tool that allows the discovery of cospectral hypergraphs. The following lemma can be obtained from \cite[Eq. (2.1)]{shao}.

\begin{lemma}\label{lem:l.count} Let $\mathcal{A}=(a_{i_1,\ldots, i_k})$ be an order $k\geq 2$ dimension $n$ tensor, and let $Q=(q_{i\,j})$ be an $n\times n$ square matrix. Then
$$(Q\mathcal{A}Q^\top)_{i_1\cdots i_k} = \sum\limits_{j_1,\dots,j_k\in[n]} a_{j_1\cdots j_k}q_{i_1\,j_1}q_{i_2\,j_2}\cdots q_{i_k\,j_k}.$$
\end{lemma}

From Lemma~\ref{lem:l.count} we obtain the following:
\begin{lemma}\label{l.sym}
Let $\mathcal{A}'=Q\mathcal{A}Q^\top$, where $\mathcal{A}$ is a tensor of dimension $n$ and $Q$ is an $n\times n$ matrix. If $\mathcal{A}$ is symmetric, then $\mathcal{A}'$ is symmetric.
\end{lemma}

Additionally, let $Q$ be a real orthogonal matrix. In \cite{shao}, Shao pointed out that $\mathcal{A}$ and $\mathcal{A}'=Q\mathcal{A}Q^\top$ are orthogonally similar tensors as defined by Qi \cite{Q2005}, which implies that they have the same set of normalized $E$-eigenvalues. 

\begin{lemma}\label{l.ortho} Let $\mathcal{A}' = Q\mathcal{A}Q^\top$, where $\mathcal{A}$ is a tensor of dimension $n$ and $Q$ is an $n\times n$ real orthogonal matrix. Then $\mathcal{A}$ and $\mathcal{A}'$ are E-cospectral.
\end{lemma}

\Cref{l.ortho} will be the key ingredient for proving the $E$-cospectrality of the hypergraphs constructed using the method described in Section \ref{sec:Ecospectralityswitching}.

\section{Indecomposable regular orthogonal matrices}
\label{sec:indecomp}
A rational orthogonal matrix $Q$ is \textit{regular} if it has constant row sum, that is, $Q\mymathbb{1} = r \mymathbb{1}$, for some $r \in \Q$. A regular orthogonal matrix $Q$ has level $\ell$ if $\ell$ is the smallest positive integer such that $\ell Q$ is an integral matrix.

A square matrix $A$ is \textit{decomposable} (\cite{wang_xu_2006,wang_xu_2010, brualdi_ryser_1991}), if there are permutation matrices $P_1, P_2$ such that 
$$
P_1 A P_2 = \begin{bmatrix} M_1 & M_2 \\ O & M_3 \end{bmatrix}
$$ 
where $M_1,M_3$ are square matrices. The matrix $A$ is \textit{indecomposable}, if it is not decomposable. 

Indecomposable regular orthogonal matrices of level $2$ and row sum $1$ are classified up to a permutation of the rows and columns. This result follows from \cite{chan_rodger_seberry} and proven in the form below by Wang and Xu in \cite[Theorem 3.1, p.~66]{wang_xu_2006} (also see \cite[Theorem 2.4, p.~73]{wang_xu_2010}, \cite[Theorem 3, p.~5]{abiad_berg_simoens}).  The focus on level $2$ is due to the observation by Wang-Xu in \cite{wang_xu_2010} that, empirically, most generalized cospectral graphs pairs are related by an orthogonal matrix of level $2$.  For hypergraphs, even such experimental insight is lacking from the literature.


\begin{theorem}[\cite{wang_xu_2006}]
\label{def:all_switching_matrices}
Up to a permutation of the rows and columns, an indecomposable regular orthogonal matrix of level $2$ and row sum $1$ is one of the following: 

\begin{tabular}{ l c }
& \\[1em]
\textit{i)} $\prescript{4}{}{R}_{gm} = \frac{1}{2} \cdot \textup{circulant}(-1,1,1,1)$ & \textit{ii)} $\prescript{n}{}{R}_{sg} := \frac{1}{2} \cdot \textup{circulant}\ev{ J, O, \ldots, O, Y}$ \\[1.5em] 

\textit{iii)} $ R_{f} := \frac{1}{2} \cdot \textup{circulant}(-1,1,1,0,1,0,0) $ & \textit{iv)}  $ R_{c} := \frac{1}{2} \cdot 
\begin{bmatrix} 	
-I & I & I & I \\
I & -Z & I & Z \\
I & Z & -Z & I \\
I & I & Z & -Z 
\end{bmatrix} $  \\[1.5em] 
& \\
\end{tabular}

where $n$ the dimension of the matrix $\prescript{n}{}{R}_{sg} $ ($n$ is assumed even), $Y = 2I_2 - J_2$ and $Z = J_2 - I_2$ (the subcripts correspond to ``Godsil-McKay'', ``sun graph'', ``Fano'' and ``Cube'' respectively).
\end{theorem}

\begin{remark}
\begin{enumerate}
\item In \cite[Equation 7, p.~11]{mao_wang_liu_qiu_2023}, a different matrix is used for sun graph switching: 
$$
\frac{1}{2} \cdot \text{circulant}\ev{ Y, O, \ldots, O, J, O}
$$
This alternative definition is indeed equivalent to the way $\prescript{n}{}{R}_{sg}$ is defined above, up to a permutation of rows and columns. 

\item $R_{f} + R_{f}^{\top} = J - 3I$
\end{enumerate}

\end{remark}


We will also consider the following regular orthogonal matrices, of level possibly higher than 2.
\begin{definition}
\label{def:gm_and_wqh_matrices}
\begin{enumerate}
\item For all even $n\geq 4$, define the matrix
$$
\prescript{n}{}{R}_{gm} := \dfrac{2}{n} \cdot J_n - I_n = \frac{2}{n} \cdot \textup{circulant}( 1 - n/2 , 1, \ldots, 1)
$$
where $n$ is the dimension of the matrix $\prescript{n}{}{R}_{gm}$.

\item For all $p\geq 1$, define 
$$
\prescript{2p}{}{R}_{wqh} := \begin{bmatrix} 
I_p - (1/p) \cdot J_p & (1/p) \cdot J_p \\ 
(1/p) \cdot J_p & I_p - (1/p) \cdot J_p
\end{bmatrix}
$$
\end{enumerate}
where $2p$ is the dimension of the matrix. The subscript corresponds to ``Wang, Qiu and Hu'' (\cite{wang_qui_hu_2019}). 
\end{definition}

Let $Q$ be square matrix. Consider $Q \mathbb{A}^n_k Q^\top$, where the product is the generalized tensor product (q.v.~\Cref{def:tensorproduct}). Assume that the variables satisfy the following symmetry equations, 

\begin{equation}
\label{eqn:symmetry}
\mathbf{a}\ev{j_1,\ldots,j_k} = \mathbf{a}\ev{j_{\sigma\ev{1}},\ldots,j_{\sigma\ev{k}}}
\end{equation}
for any $\sigma\in \mathrm{Sym}\ev{[k]}$ and $j_1,\ldots,j_k$. Then, we can use permanents to express $Q^\top \mathbb{A}^n_k Q$ as follows. 

\begin{align*}
& \ew{  Q \mathbb{A}^n_k Q^\top  }_{i_1,\ldots,i_k} = \sum_{1\leq j_1, \ldots ,j_k \leq n} \mathbf{a}\ev{j_1,\ldots,j_k} q_{i_1,j_1} \cdots q_{i_k,j_k} && \text{ by \Cref{lem:l.count}} \\ 
& = \sum_{ \{ j_1,\ldots,j_k\} \in \binom{ [n] }{k} } \mathbf{a}\ev{j_1,\ldots,j_k} \sum_{ \sigma \in \mathrm{Sym}\ev{[k]} }   q_{ i_1, j_{\sigma\ev{1}} }  \cdots  q_{ i_k,j_{\sigma\ev{k}} } &&\text{ since $\mathbb{A}^n_k$ is symmetric, by (\ref{eqn:symmetry}) } \\ 
& = \sum_{ \{ j_1,\ldots,j_k\} \in \binom{ [n] }{k} } \mathbf{a}\ev{j_1,\ldots,j_k} \mathrm{per}\ew{ Q\big\vert_{ \{ i_1,\ldots,i_k\} \times \{ j_1,\ldots,j_k \} } } 
\end{align*}
where 
$Q\big\vert_{ \{ i_1,\ldots,i_k\} \times \{ j_1,\ldots,j_k \} }  = \ev{ q_{i,j}}_{\ev{ i,j } \in \{ i_1,\ldots,i_k\} \times \{ j_1,\ldots,j_k \} } $ is the $k \times k $ matrix obtained from $Q$ by extraction of the entries with indices $\ev{ i,j } \in \{ i_1,\ldots,i_k\} \times \{ j_1,\ldots,j_k \}$. In particular, we may substitute $Q^\top$ for $Q$ to obtain:
$$
\ew{  Q^\top \mathbb{A}^n_k Q  }_{i_1,\ldots,i_k} = \sum_{ \{ j_1,\ldots,j_k\} \in \binom{ [n] }{k} } \mathbf{a}\ev{j_1,\ldots,j_k} \mathrm{per}\ew{ (Q^{\top})\big\vert_{ \{ i_1,\ldots,i_k\} \times \{ j_1,\ldots,j_k. \} } }. 
$$


\section{Switching operations for \texorpdfstring{$k$}{k}-uniform hypergraphs}
\label{sec:Ecospectralityswitching}

Let $R$ be one of the regular orthogonal matrices given in \Cref{def:all_switching_matrices} and \Cref{def:gm_and_wqh_matrices}, of dimension $s$. Let $Q = R \oplus I_{n-s}$, where $n\geq s$. For the rest of the manuscript, we use the notation $R$ and $Q$ to refer to a fixed choice of these matrices. Let $C = \{ v_1,\ldots,v_s\}$ and $D = \{ v_{s+1},\ldots,v_n\}$ be fixed disjoint sets. The set $C$ is the \textit{switching} set. We will consider hypergraphs $G$ and $H$ on the same vertex set $C\sqcup D$. 

Let $\mathcal{V}\ev{Q}$ be the set of zero-one vectors $\mathbf{v}$ such that $Q^\top \mathbf{v}$ is also a zero-one vector.

For each $k\geq 2$, let $\mathcal{B}^{k}_Q$ be the set of $k$-uniform hypergraphs $G$ such that $Q^\top \cdot \mathcal{A}_G \cdot Q = \mathcal{A}_H $ for some $k$-uniform hypergraph $H$. If $G \in \mathcal{B}^{k}_Q$, then let $t\ev{G}$ be the hypergraph such that $Q^\top \cdot \mathcal{A}_G \cdot Q = \mathcal{A}_{t\ev{G}} $. If the vertex sets of $G$ and $H$ are clear from context, we may sometimes write $t\ev{G}$ to refer to the edge set of the hypergraph in question.

A hypergraph of rank $1$ with vertex set $C$ is defined as $\ev{C,E}$ for some subset $E\subseteq \binom{C}{1}$. Equivalently, given a $1$-graph $G$ on a vertex set $C$, then the edge set of $G$ is a subset of $C$. The adjacency matrix of $G$ is a zero-one vector of length $|C|$. We extend the definition of $\mathcal{B}^{k}_Q$ to $k=1$ by defining $\mathcal{B}^{1}_Q := \mathcal{V}\ev{Q}$. 

We will need pairs $\ev{ G, t\ev{G}}$ such that $G\in \mathcal{B}^{k}_Q$, where $Q$ is one of the matrices defined in \Cref{sec:indecomp}. For $k=1,2$, the hypergraphs $\mathcal{B}^{k}_Q$ have been described in the literature, summarized in \Cref{table:exhaustive_list}. 
\begin{table}[h!]
\centering
\begin{tabular}{ | c | c | c |}
\hline
$Q$ & $k$ & Description of $\mathcal{B}^{k}_Q$ found in: \\[2mm]
\hline
$R_f$ & $1$ & \cite[Lemma 7, p.~9]{abiad_haemers} \\[1mm] 
& $2$  & \cite[Lemma 8, p.~10]{abiad_haemers}  \\ [1mm]
\hline
$R_c$   & $1$ & \cite[Lemma 30, p.~22]{abiad_berg_simoens}\\ [1mm] 
& $2$ & \cite[Table 2, p.~23]{abiad_berg_simoens}\\ [1mm] 
\hline
$\prescript{n}{}{R}_{sg} $  & $1$ & \cite[Lemma 6, p.~7]{abiad_berg_simoens}, \cite[Theorem 6.1, p.~15]{mao_wang_liu_qiu_2023} \\ [1mm] 
& $2$ & \cite[Theorems 5,7,8]{abiad_berg_simoens}, \cite[Theorem 4.1, p.~11]{mao_wang_liu_qiu_2023}  \\ [1mm] 
\hline
$\prescript{n}{}{R}_{gm}$ & $1 \& 2$ & \cite[Theorem 1.1, p.~156]{wang_qui_hu_2019} \\[1mm] \hline
$\prescript{2p}{}{R}_{wqh} $ & $1$ & \cite[Theorem 3.4, p.~163]{wang_qui_hu_2019} \\[1mm]
& $2$  &  \cite[Theorem 3.1, p.~160]{wang_qui_hu_2019} \\ [1mm] 
\hline
\end{tabular}
\caption{References for the list of known graphs (or vectors) $G$ such that $Q^\top \mathcal{A}_G Q$ is a zero-one matrix (or vector). }
\label{table:exhaustive_list}
\end{table}

In the next proposition, we determine $\mathcal{B}^{k}_Q$ for $k=3$ and $Q = R_f$, as well as collecting some previously known cases for reference. We use edge sets to define the hypergraphs in question. (In the case of $\prescript{2p}{}{R}_{wqh}$, we let $C = C_1 \sqcup C_2$, where $|C_i| = p$, for some $p\geq 1$. In the case of $\prescript{n}{}{R}_{sg}$, we assume $C = C_1 \cup \ldots \cup C_{m}$, for an odd integer $m\geq 3$. The definition of $F_1$ and $F_2$ can be found in \Cref{sec:conseqs} below.)

\begin{proposition}
\label{prop:exhaustive_data}
\leavevmode
The pairs $ \{ \ev{ E(G), E(t\ev{G}) }: G \in \mathcal{B}^{k}_Q \}$ are given by:
\begin{enumerate}
\item[(i)] If $k=1$ and $Q = \prescript{n}{}{R}_{gm}$: 

$ \{ (\emptyset,\emptyset), (\binom{C}{1} ,\binom{C}{1}) \} \cup \{ (\binom{X}{1}, \binom{C\setminus X}{1}) :  X \subseteq C \text{ and } |X| = |C|/2 \} $  

\item[(ii)] If $k=1$ and $Q = \prescript{2p}{}{R}_{wqh}$: 

$ \{(\binom{X}{1},\binom{X}{1}): X \subseteq C_1\cup C_2 \text{ and } |X \cap C_1| = |X \cap C_2| \} \cup \{ (\binom{C_1}{1}, \binom{C_2}{1}) \} $

\item[(iii)] If $k=1$ and $Q = R_f$: 

$\{ (\emptyset,\emptyset), ( \binom{C}{1}, \binom{C}{1} ) \} \cup \{ \ev{ \ell_i^1, \mathcal{O}_i^1 }: i=0,\ldots,6 \} \cup \{ \ev{ \binom{C}{1} \setminus \ell_i^1, \binom{C}{1} \setminus \mathcal{O}_i^1 }  : i=0,\ldots,6 \}$

\item[(iv)] If $k=3$ and $Q = R_f$: 

$ \{ (\emptyset,\emptyset), \ev{ F_1 , F_2 } \}$, for $k=3$ and $Q = R_f$. 

\item[(v)] If $k=1$ and $Q = \prescript{n}{}{R}_{sg}$: 

$ \{ \ev{ \binom{X}{1}, \binom{X}{1} } : |X \cap C_i| \equiv 0 \pmod{2},\ \forall i \} \cup \{ \ev{ \binom{X}{1}, \binom{\pi(X)}{1} }: |X\cap C_i| = 1, \forall i \}$

where $C_i = \{ v_{2i-1}, v_{2i} \}$ for each $i=1,\ldots,m$ and $\pi\ev{ v_j } = v_{j-2 \pmod{2m}}$, for each $j=1,\ldots,2m$. ($\pi\ev{v_2} = v_{2m} \in C$.)
\end{enumerate}
\end{proposition}
\begin{proof}
For $k=1,2$, we refer to \Cref{table:exhaustive_list}. For $k=3$, we used Sagemath and applied exhaustive search to completely determine $\mathcal{B}^{3}_{R_f}$ (c.f.~\cite{github}). 
\end{proof}

\section{Main Switching Theorem}

Let $R$ and $Q$ be square matrices as in the previous section. For a polynomial $f\ev{x}$ and a constant $a$, we write $f\ev{x} [ x = a ]$, for the polynomial obtained by replacing $x$ with $a$. 

Recall that $\mathbb{A}^n_k$ denotes the generic tensor of rank $k$ and dimension $n$. For each hypergraph $G$, let us define the indicator function 
$$\chi\ev{ \{ i_1, \ldots, i_k \} \in E\ev{G} } := \begin{cases} 1 & \text{ if } \{ i_1, \ldots, i_k \} \in E\ev{G}, \\ 0 &\text{ otherwise. } \end{cases}
$$
Then, we have $\ev{ \mathcal{A}_{G} }_{i_1,\ldots,i_k} = \ev{\mathbb{A}^n_k}_{i_1,\ldots,i_k} [ \mathbf{a}\ev{j_1,\ldots,j_k} = \chi\ev{ \{ j_1, \ldots, j_k \} \in E\ev{G} }, \forall j_1,\ldots,j_k ]$.

For two edge sets $X,Y$, let us write 
$$
X\odot Y = \{ e \cup f: e\in X, \ f \in Y \}.
$$
In particular, we have $\emptyset \odot Y = \emptyset = X\odot \emptyset$.  
\begin{theorem}
\label{thm:main_theorem}
Let $G = \ev{V\ev{G},E\ev{G}}$ be a $k$-uniform hypergraph with $k\geq 2$. Let $C=\{ v_1,\ldots,v_s \}$ be a subset of $V\ev{G}$. Let $D:= V\ev{G}\setminus C$. Assume that for all $r\in \{1,2,\ldots,k\}$ and for all $A \in \binom{D}{k-r} $, there is some hypergraph $L_r^A \in \mathcal{B}^{r}_{R}$ induced on the set $C$ ($E\ev{L_r^A}$ can be empty) such that 
\begin{equation}
\label{eqn:to_refer}
\text{$\left\{ e \cup A \in E\ev{G}: e\in \binom{C}{r} \right\} = E\ev{L_r^A} \odot \{A\}$. }
\end{equation}
To construct a $k$-uniform hypergraph $H$, for each $r$ and each $A \in \binom{D}{k-r}$, remove the edges $E\ev{L_r^A} \odot \{A\}$ and add the edges $ E\ev{t\ev{L_r^A}} \odot \{A\}$. Then, $G$ and $H$ are $E$-cospectral.
\end{theorem}

\begin{proof}
By the construction of $H$, 
\begin{equation}
\label{eqn:assumption_implies_2}
\text{ $\left\{ e \cup A \in E\ev{H}: e\in \binom{C}{r} \right\} = E\ev{t\ev{L_r^A}} \odot \{A\}$}
\end{equation}
for each $r$ and $A \in \binom{D}{k-r}$. By the definition of $t\ev{L_r^A}$, we know that, for each $A$,
\begin{align*}
& R^\top \mathcal{A}_{L_r^A} R = \mathcal{A}_{t\ev{L_r^A}}.
\end{align*}
Given a fixed $A$, let $r:= k-|A|$. We can see that
\begin{align*}
\ev{ R^\top \mathbb{A}^{s}_r R }_{i_1,\ldots,i_k} [ \mathbf{a}\ev{j_1,\ldots,j_r}& = \chi\ev{ \{ j_1, \ldots, j_r \} \in E\ev{L_r^A} }, \forall j_1,\ldots,j_r ]\\
& \hspace{-2cm}= R^\top \mathcal{A}_{L_r^A} R \\
& \hspace{-2cm}= \mathcal{A}_{t\ev{L_r^A}}  \\
& \hspace{-2cm} = \ev{ \mathbb{A}^{s}_r }_{i_1,\ldots,i_k} [ \mathbf{a}\ev{j_1,\ldots,j_r} = \chi\ev{ \{ j_1, \ldots, j_r \} \in E\ev{t\ev{L_r^A}} }, \forall j_1,\ldots,j_r ].
\end{align*}

Recall that $Q = R\oplus I_{n-s}$. To see that $Q^\top \mathcal{A}_G Q = \mathcal{A}_H$, it is enough to show that their values agree for each index set $\{i_1,\ldots,i_k\}$. Let $1\leq i_1 < \ldots < i_k \leq n$ be chosen. Without loss of generality, we may assume that $1\leq i_1 < \ldots < i_r \leq s < i_{r+1} < \ldots < i_k \leq n$, for some $r$. Define $A:= \{ i_{r+1}, \ldots, i_{k}\}$. We know that
\begin{align*}
& (Q^\top \mathbb{A}^n_k Q)_{ i_1,\ldots,i_k } = \sum_{ \{ j_1,\ldots,j_k\} \in \binom{ [n] }{k} } \mathbf{a}\ev{j_1,\ldots,j_k}\mathrm{per}\ew{ Q^{\top}\big\vert_{ \{ i_1,\ldots,i_k\} \times \{ j_1,\ldots,j_k \} } } 
\end{align*}
Since $(Q^{\top})_{ a,b } = \delta_{a,b}$, for any $s<a,b$, where $\delta$ is the Kronecker symbol, it follows that
\begin{equation}
\label{eqn:assumption_4}
(Q^\top \mathbb{A}^n_k Q)_{ i_1,\ldots,i_k } =  \sum_{ \{ j_1,\ldots,j_r\} \in \binom{ [s] }{r} } \mathbf{a}\ev{j_1,\ldots, j_r, i_{r+1},\ldots ,i_k}\mathrm{per}\ew{ R^{\top}\big\vert_{ \{ i_1,\ldots,i_r\} \times \{ j_1,\ldots,j_r \} } } 
\end{equation}
Hence, we obtain
\begin{align*}
&(Q^\top \mathbb{A}^n_k Q)_{ i_1,\ldots,i_k }[ \mathbf{a}\ev{j_1,\ldots,j_k} = \chi\ev{ \{ j_1, \ldots, j_k \} \in E\ev{G} }, \forall j_1,\ldots,j_k ] \\
& =(Q^\top \mathbb{A}^n_k Q)_{ i_1,\ldots,i_k }
[\mathbf{a}\ev{j_1,\ldots,j_r,i_{r+1},\ldots,i_k} = \chi\ev{ \{ j_1,\ldots,j_r,i_{r+1},\ldots,i_k \} \in E\ev{G} }, \forall j_1,\ldots,j_r]
\end{align*}
By assumption (\ref{eqn:to_refer}), we have
\begin{align*}
&(Q^\top \mathbb{A}^n_k Q)_{ i_1,\ldots,i_k }[ \mathbf{a}\ev{j_1,\ldots,j_k} = \chi\ev{ \{ j_1, \ldots, j_k \} \in E\ev{G} }, \forall j_1,\ldots,j_k ] \\
& = (Q^\top \mathbb{A}^n_k Q)_{ i_1,\ldots,i_k } [ \mathbf{a}\ev{j_1,\ldots, j_k} = \chi\ev{ \{ j_1, \ldots, j_r \} \in E\ev{L_r^A} } \delta\ev{j_{r+1},i_{r+1}} \cdots  \delta\ev{j_{k},i_{k}}, \forall j_1,\ldots,j_r  ] 
\end{align*}

On the other hand, 
\begin{align*}
& \ev{ R^\top \mathbb{A}^{s}_r R }_{i_1,\ldots,i_r} = \sum_{ \{ j_1,\ldots,j_r\} \in \binom{ [s] }{r} } \mathbf{a}\ev{j_1,\ldots, j_r}\mathrm{per}\ew{ R^\top\big\vert_{ \{ i_1,\ldots,i_r\} \times \{ j_1,\ldots,j_r \} } } 
\end{align*}
which implies, by (\ref{eqn:assumption_4}) that
\begin{equation}
\label{eqn:assumption_3}
(Q^\top \mathbb{A}^n_k Q)_{ i_1,\ldots,i_k } [ \mathbf{a}\ev{j_1,\ldots, j_r, i_{r+1},\ldots ,i_k} = \mathbf{a}\ev{j_1,\ldots, j_r}, \forall j_1,\ldots,j_r  ] = \ev{ R^\top \mathbb{A}^{s}_r R }_{i_1,\ldots,i_r}
\end{equation}

Then, 
\begin{align*}
& \ev{ Q^\top \mathcal{A}_G Q }_{ i_1,\ldots,i_k }  = (Q^\top \mathbb{A}^n_k Q)_{ i_1,\ldots,i_k } [ \mathbf{a}\ev{j_1,\ldots,j_k} = \chi\ev{ \{ j_1, \ldots, j_k \} \in E\ev{G} }, \forall j_1,\ldots,j_k ] \\
& = (Q^\top \mathbb{A}^n_k Q)_{ i_1,\ldots,i_k } [ \mathbf{a}\ev{j_1,\ldots, j_k} = \chi\ev{ \{ j_1, \ldots, j_r \} \in E\ev{L_r^A} } \delta\ev{j_{r+1},i_{r+1}} \cdots  \delta\ev{j_{k},i_{k}}, \forall j_1,\ldots,j_k ] \\ 
& = \ev{ R^\top \mathbb{A}^{s}_r R }_{i_1,\ldots,i_r} [ \mathbf{a}\ev{j_1,\ldots,j_r} = \chi\ev{ \{ j_1, \ldots, j_r \} \in E\ev{L_r^A} }, \forall j_1,\ldots,j_r ] \\
& \text{ by (\ref{eqn:assumption_3}) } \\
& =  \ev{ \mathbb{A}^{s}_r }_{i_1,\ldots,i_r} [ \mathbf{a}\ev{j_1,\ldots,j_r} = \chi\ev{ \{ j_1, \ldots, j_r \} \in E\ev{t\ev{L_r^A}} }, \forall j_1,\ldots,j_r ]\\ 
& =  \ev{ \mathbb{A}^{n}_k }_{i_1,\ldots,i_k} [ \mathbf{a}\ev{j_1,\ldots,j_k} = \chi\ev{ \{ j_1, \ldots, j_k \} \in E\ev{H} }, \forall j_1,\ldots,j_k ]\\
& \text{ by (\ref{eqn:assumption_implies_2}), in a similar fashion} \\
& = \ev{ \mathcal{A}_H }_{ i_1,\ldots,i_k }.
\qedhere\end{align*}
\end{proof}

Historically, most switching results describe a process of removing and replacing edges, therefore in order to unify nomenclature we restate \Cref{thm:main_theorem} below in slightly different, more concise language by generalizing the classical notion of the ``link'' of vertices in a hypergraph. For each $C\subseteq V\ev{G}$ and $A\subseteq V\ev{G}\setminus C$, let the \textit{link of $A$ in $C$} be the hypergraph $G[C;A]$ with vertex set $C$ and edge set $\{ f\in \binom{C}{k-|A|} : f\cup A \in E\ev{G} \} $. This notation allows us to express the main theorem as follows: 

\begin{corollary}
\label{cor:main_theorem_2}
Fix $R\in \Q^{s\times s}$ and $Q := R\oplus I_{n-s}$. Let $G$ be a $k$-uniform hypergraph with $k\geq 2$. Let $C=\{ 1,\ldots, s \}$ be a subset of $V\ev{G}$. Let $D := V\ev{G}\setminus C $. Assume that for each $A \subseteq D$, we have $L^{A} := G[C;A] \in \mathcal{B}^{k-|A|}_{R}$. Then, define $H$ on $V\ev{G}$ by 
$$
E\ev{H} := \bigcup_{A \subseteq D}  E\ev{t\ev{L^{A}}} \odot \{A\}.
$$
Then, $G$ and $H$ satisfy $Q^\top \mathcal{A}_{G} Q = \mathcal{A}_{H}$, and $G$ and $H$ are therefore $E$-cospectral.
\end{corollary}

\section{Consequences of Theorem \ref{thm:main_theorem}}
\label{sec:conseqs}
In this section, we use the general method from the previous section to recover several known switching methods for obtaining $E$-cospectral hypergraphs (GM switching \cite{bzw2014} and WQH-switching \cite{abiad_khramova}). Afterwards, we show that our general \Cref{thm:main_theorem} also provides several new switching methods. In the literature, hypergraph switching methods are stated such that the induced hypergraph on the switching set is empty, but this is not required as shown by \Cref{thm:main_theorem}.

The \emph{degree} of a vertex $u\in V(G)$ is the number of edges that contain $u$. For any edge $\{ u_1,\ldots, u_k\} \in E\ev{G}$, we say that $u_1$ is a \textit{neighbour} of $\{ u_2,\ldots, u_k\}$. The \textit{neighbourhood} of $\{ u_2,\ldots, u_k\}$ is defined as $\Gamma(u_2,\ldots, u_k):= \{ x\in V\ev{G}: \{x,u_2,\ldots, u_k \} \in E\ev{G}\}$.
\subsection{GM hypergraph switching}

The GM-switching was generalized to $k$-uniform hypergraphs, for $k\geq 2$, in \cite[Theorem 3.1, p.~4]{bzw2014}. We will provide an alternative proof using \Cref{thm:main_theorem}. Let $Q = \prescript{s}{}{R}_{gm} \oplus I_{n-s}$, for some even $s\geq 4$ and $n\geq s$ (q.v.~\Cref{def:gm_and_wqh_matrices}).
\begin{corollary}[\cite{bzw2014}]
\label{cor:E_gm_switch}
Let $G$ be a $k$-uniform hypergraph, for some $k\geq 2$. Assume that the vertex set has a partition $V\ev{G} = C \sqcup D$ with $|C|=s$. Furthermore, assume that $G$ satisfies the conditions,
\begin{enumerate}
\item For all $e\in E\ev{G}$, we have $|e \cap C|\leq 1$. 
\item For any $k-1$ distinct vertices $u_2,\ldots, u_k$ from $D$, we have $|\Gamma\ev{ u_2,\ldots, u_k } \cap C| \in \{ 0, (1/2)\cdot |C|, |C| \}$.
\end{enumerate}
To construct a hypergraph $H$, for all $\{ u_2, \ldots, u_k \} \subseteq D$ with $ (1/2)|C| $ many neighbours in $C$, remove the edges $\{ \{ x, u_2,\ldots,u_k\} \in E\ev{G} : x\in C\}$ and add the edges $\{ \{ x, u_2,\ldots,u_k\} \notin E\ev{G}: x\in C \} $. Then, we have $Q^\top \mathcal{A}_G Q = \mathcal{A}_H$ and in particular, the hypergraphs $G$ and $H$ are E-cospectral. 
\end{corollary} 
\begin{proof}
By the first condition that $G$ satisfies, it is enough to consider $r=1$ in \Cref{thm:main_theorem}. Take any subset $A := \{ u_2, \ldots, u_k \} \subseteq D$, and consider the $1$-graph $L_1^{A}$ with vertex set $C$ and edge set $E\ev{L_1^{A}} := \{ \{ x \}: x\in \Gamma\ev{ u_2,\ldots, u_k } \cap C\}$. By assumption, we know that $|E\ev{L_1^{A}}| = (1/2)\cdot |C|$. By \Cref{prop:exhaustive_data} Part \textit{(i)}, we have $L_1^{A} \in \mathcal{V}\ev{\prescript{s}{}{R}_{gm}} $ and $E\ev{ t\ev{ L_1^{A} }} = \binom{C}{1} \setminus E\ev{ L_1^{A} } = \{ \{ x \}: x\in C \setminus \Gamma\ev{ u_2,\ldots, u_k } \} $. 

We can see that (\ref{eqn:to_refer}) of \Cref{thm:main_theorem} is satisfied, and also that $H$ is constructed exactly by removing the edges $E\ev{L_1^{A}} \odot \{A\}$ and adding the edges $E\ev{t\ev{ L_1^{A} }} \odot \{A\}$, so the statement follows.
\end{proof}

\subsection{WQH hypergraph switching}

In \cite[Theorem 2.6, p.~4]{abiad_khramova}, the WQH-switching method of \cite{wang_qui_hu_2019} was generalized to hypergraphs. We will provide an alternative proof below, using our main theorem. Let $Q := \prescript{2p}{}{R}_{wqh} \oplus I_{n-2p}$, for some $1\leq p$ and $2p\leq n$ (q.v.~\Cref{def:gm_and_wqh_matrices}). Note that $Q^\top = Q$.
\begin{corollary}[\cite{abiad_khramova}]
\label{cor:E_WQH_switch}

Let $G$ be a $k$-uniform hypergraph, for some $k\geq 2$. Assume that the vertex set has a partition $V\ev{G} = C_1 \sqcup C_2 \sqcup D$ with $|C_i| = p $. Furthermore, assume that $G$ satisfies the conditions,
\begin{enumerate}
\item For all $e\in E\ev{G}$, we have $|e \cap (C_1 \cup C_2) |\leq 1$. 
\item For any $k-1$ distinct vertices $u_2,\ldots, u_k$ from $D$, we have $\Gamma\ev{ u_2,\ldots, u_k } \cap (C_1 \cup C_2) \in \{ C_1, C_2 \}$ or $|\Gamma\ev{ u_2,\ldots, u_k } \cap C_1| = |\Gamma\ev{ u_2,\ldots, u_k } \cap C_2|$.
\end{enumerate}
To construct a hypergraph $H$, for all $\{ u_2, \ldots, u_k \} \subseteq D$ such that $\Gamma\ev{ u_2,\ldots, u_k } \cap (C_1 \cup C_2) \in \{ C_1, C_2 \}$, switch the adjacency of $\{u_2,\ldots, u_k\}$ for all $u_1 \in C_1 \cup C_2$. Then, we have $Q^\top \mathcal{A}_G Q = \mathcal{A}_H$ and in particular, the hypergraphs $G$ and $H$ are E-cospectral. 
\end{corollary}
\begin{proof}
By the first condition that $G$ satisfies, it is enough to consider $r=1$ in \Cref{thm:main_theorem}. Take any subset $A := \{ u_2, \ldots, u_k \} \subseteq D$. Without loss of generality, assume that $\Gamma\ev{ u_2,\ldots, u_k } \cap (C_1 \cup C_2) = C_1 $. Define the $1$-graph $L_1^{A}$ with vertex set $C_1 \cup C_2$ and edge set $E\ev{L_1^{A}} := \{ \{ x \}: x\in C_1\}$. By \Cref{prop:exhaustive_data} Part \textit{(ii)}, we know that $L_1^{A} \in \mathcal{V}\ev{\prescript{2p}{}{R}_{wqh}} $ and $E\ev{ t\ev{ L_1^{A} }} = \binom{C}{1} \setminus E\ev{L_1^{A}} = \{ \{ x \}: x\in C_2 \} $. We can see that (\ref{eqn:to_refer}) of \Cref{thm:main_theorem} is satisfied, and also that $H$ is constructed exactly by removing the edges $E\ev{L_1^{A}} \odot \{A\}$ and adding the edges $E\ev{t\ev{ L_1^{A} }} \odot \{A\}$, so the statement follows.
\end{proof}

\subsection{Sun hypergraph switching}
We have a generalization of sun graph switching as found in \cite[Theorem 5, p.~6]{abiad_berg_simoens}, which is equivalent to the original description in \cite{mao_wang_liu_qiu_2023} (the matrices used for $R_{sg}$ in these descriptions are different, but equivalent up to a reordering of rows and columns.) 

Recall that a \textit{sun graph} is a graph on $2m$ vertices, obtained by adding a pendant edge to each vertex of a cycle of length $m\geq 3$. In the theorem below, the induced graph on $C$ is a ``generalized'' sun graph, which is obtained from a sun graph by adding complete bipartite graphs $K_{2,2}$ based on the given rules below. The switching operation produces another generalized sun graph. 

As in \Cref{prop:exhaustive_data}, we assume $V\ev{G} = C_1 \sqcup \ldots \sqcup C_{m} \sqcup D $, where $C_i = \{ v_{2i-1}, v_{2i}\}$, for each $i=1,\ldots, m$ and $m\geq 3$ is odd. Recall that $\pi\ev{v_j} = v_{j-2 \pmod{2m}}$, for each $j=1\ldots,2m$ ($\pi\ev{v_2} = v_{2m}\in C_m$.); in particular, $\pi\ev{C_i} = C_{i-1}$, for each $i$. 
\begin{corollary}
Let $k\geq 2$ and let $G$ be a $k$-uniform hypergraph. Assume that:
\begin{enumerate}
\item Every $(k-1)$-subset of $D$ has the same number of neighbors in $C_1,\ldots,C_m$ modulo $2$.
\item For each $A\in \binom{D}{k-2}$ and for all $1\leq i < j \leq i + \frac{m-1}{2}$, the edge set of the link $G[C_i \cup C_j;A]$ is:
\begin{align*}
\emptyset \text{ or } C_i \times C_j 				\ (\text{empty or complete bipartite})	 && \hspace{-3mm}  \text{ if } j < i + \frac{m-1}{2} \\ 
\{\{v_{2i-1}, x\} : x\in C_j\} \text{ or } \{\{v_{2i},x\}: x\in C_j\} 					     && \hspace{-3mm} \text{ if } j = i +\frac{m-1}{2}
\end{align*}
\end{enumerate}
To construct a hypergraph $H$, 
\begin{enumerate}
\item For every $A\in \binom{D}{k-1}$ such that $A$ has exactly one neighbour $v_{s_i} \in C_i$ ($s_i\in \{2i-1,2i\}$) for each $i$, remove the edges $ \{ v_{s_i} \} \odot \{ A \}$ and add the edges $ \pi\ev{ v_{s_i} } \odot \{A\}$. 
\item For all $A\in \binom{D}{k-2}$ and for each $i=1,\ldots,m$, remove the edges of $A\cup C_i \cup C_{i+\frac{m-1}{2}}$ and replace them with those of the induced subgraph on $A\cup C_i \cup C_{i+\frac{m+1}{2}}$, in other words, add the edges $\{A\cup\{v_{2i - 2 + m}, x\} : x\in C_i\}$ if $\{A\cup\{v_{2i + m}, x\} : x\in C_i\}$ are edges of $G$, or add $\{A\cup\{v_{2i - 1 + m}, x\} : x\in C_i\}$ if $\{A\cup\{v_{2i + m + 1}, x\} : x\in C_i\}$ are edges of $G$. 
\end{enumerate}
Then, $G$ and $H$ are E-cospectral.
\end{corollary}
\begin{proof}
We need to consider two cases. If $k=1$, then the neighbourhood of any $k-1$-subset of $D$ is an element of $\mathcal{V}\ev{R_{sg}}$ and its replacement is a zero-one vector, by \Cref{prop:exhaustive_data}. If $k=2$, the induced subgraph on $C$ is in $\mathcal{B}^2_{R_{sg}}$ and its replacement is a graph, by \cite[Theorem 5, p.~6]{abiad_berg_simoens}. It follows, by \Cref{thm:main_theorem} that $G$ and $H$ are E-cospectral. 
\end{proof}

\begin{example}
As an example, consider $3$-uniform hypergraphs $G,H$ such that $V\ev{G} = V\ev{H} = \{ 0, 1,\ldots, 8, 9\}\cup \{x,y\}$ and $Q = \prescript{10}{}{R}_{sg} \oplus I_2$ and $Q^\top \mathcal{A}_G Q = \mathcal{A}_H$.
\begin{align*}
E\ev{G} & = \{ 18x, 25x, 26x, 28x, 03x, 47x, 48x, 04x, 69x, 06x, 1xy, 3xy, 5xy, 7xy, 0xy \} \\
E\ev{H} & = \{ 16x, 26x, 27x, 28x, 38x, 48x, 49x, 04x, 05x, 06x, 9xy, 1xy, 3xy, 5xy, 8xy \}
\end{align*}
\end{example}

\subsection{Fano hypergraph switching}

Let $C=\{ v_1, \ldots, v_7 \}$ be a fixed set. For each $i$, let $\pi \ev{ i }$ be the unique $j \in \{ 1, \ldots, 7\}$ such that $j \equiv i+1 \pmod{7} $.
\begin{definition}[Fano Plane as a 3-Graph]\leavevmode

\begin{enumerate}
\item Define $\ell^3 = \{ v_1, v_2, v_4 \}$ and $\mathcal{O}^3 = \{ v_3, v_5, v_6 \}$. 
\item Let $\ell_i := \pi^{i} \ev{ \ell^3 }$ and $\mathcal{O}_i := \pi^{i} \ev{ \mathcal{O}^3 }$ for $i=0,\ldots,6$  be the ``lines'' and ``ovals''. (Note that $C = \ell_i \sqcup \mathcal{O}_{i} \sqcup \{ v_{i+7} \}$, for each $i \pmod{7}$.)
\end{enumerate} 
\end{definition}
Let us write $ F_1 = \{\ell_i: i = 0,\ldots,6 \} $ and $F_2 = \{ \mathcal{O}_i : i = 0,\ldots,6 \}$. Then, $\ev{C, F_1}$ and $\ev{C, F_2}$ both define $3$-uniform Fano planes.

\begin{definition}[Edges of the Fano Plane as $1$-Graphs]\leavevmode
\begin{enumerate}
\item Let $\ell^1 := \{ \{v_1\},\{v_2\},\{v_4\} \}$ and $\mathcal{O}^1 = \{  \{v_3\},\{v_5\},\{v_6\} \}$.
\item Let $\ell_i^1 := \pi^{i} \ev{ \ell^1 }$ and $\mathcal{O}_i^1 := \pi^{i} \ev{ \mathcal{O}^1 }$ for $i=0,\ldots,6$. 
\end{enumerate} 
Then, the pairs $\ev{C, \ell_i^1}$ and  $\ev{C, \mathcal{O}_i^1}$ define $1$-graphs. 
\end{definition}

Recall that \cite[Theorem 25, p.~19]{abiad_berg_simoens} describes Fano switching for graphs of rank $2$. Here, we provide an analogue for $3$-uniform hypergraphs. 
\begin{corollary}
Let $G$ be a $3$-uniform hypergraph such that for all $e \in E\ev{G}$, we have $|e\cap C| \leq 1$. Let $C=\{ v_1,\ldots,v_7 \}$ be a subset of $V\ev{G}$. Let $D:= V\ev{G}\setminus C$. 
Suppose that:
\begin{enumerate}
\item[\textit{(i)}] the edge set of the induced subgraph $L$ on $C$ is empty or $F_1 \in \mathcal{B}^{3}_{R_f}$. 
\item[\textit{(ii)}] for any distinct pair of vertices $u_2, u_3$ from $D$, and for any $i=0,\ldots,6$, we have $\Gamma(u_2,u_3) \cap C \in \{ \emptyset, C \} \cup \{ \ell_i \}_{i=0}^{6} \cup \{ C\setminus \ell_i \}_{i=0}^{6}  $. 
\end{enumerate}
To construct a hypergraph $H$, replace the induced hypergraph $L$ with $t\ev{L}$; and for any $\{ u_2, u_3 \} \subseteq D$:
\begin{enumerate}
\item In case $\Gamma_G(u_2, u_3) \cap C = \ell_i$, then remove the edges $ \{ \{ x, u_2, u_3\} : x\in \ell_i \}$ and add the edges $ \{ \{ x, u_2, u_3\} : x\in \mathcal{O}_{i} \}$. 
\item In case $\Gamma_G(u_2, u_3) \cap C = C \setminus \ell_i$, then remove the edges $ \{ \{ x, u_2, u_3\} : x\in C \setminus \ell_i \}$ and add the edges $ \{ \{ x, u_2, u_3\} : x\in C \setminus \mathcal{O}_{i} \}$. 
\end{enumerate}
Then, $G$ and $H$ are E-cospectral. 
\end{corollary}
\begin{proof}
For $r=3$, we only need to consider $A = \emptyset \in \binom{D}{0} = \{ \emptyset \}$. Then, define $L_3^{A}$ to be the induced hypergraph $F_1$ on $C$ given by the hypothesis. For $r=1$, let $A := \{ u_2, u_3 \} \in \binom{D}{2}$ be any subset. By hypothesis, we consider four cases and in each case, we define a $1$-graph $ L_1^{A} \in \mathcal{B}^1_{R_f} $ with vertex set $C$ and a certain edge set. Afterwards, we refer to \Cref{prop:exhaustive_data} for the calculation of $t\ev{L_1^A}$. 
\begin{itemize}
\item $\Gamma(u_2,u_3) \cap C = C$. Then, define $L_1^{A}$ with edge set $\binom{C}{1}$. Then, we have $t\ev{L_1^{A}} = L_1^{A}$. 
\item $\Gamma(u_2,u_3) \cap C = \emptyset$. Then, define $ L_1^{A} $ with empty edge set. In this case, $t\ev{L_1^{A}} = L_1^{A}$ as well.
\item $\Gamma(u_2,u_3) \cap C = \ell_i$ for some $i$. Then, define $ L_1^{A} $ with edge set $\ell^1_i$. Note that $t\ev{L_1^{A}} = \mathcal{O}_{i}$. 
\item $\Gamma(u_2,u_3) \cap C = \binom{C}{1} \setminus \ell_i$ for some $i$. Then, define $ L_1^{A} $ with edge set $\binom{C}{1} \setminus \ell^1_i$. Note that $t\ev{L_1^{A}} = \binom{C}{1} \setminus \mathcal{O}_{i}$. 
\end{itemize} 
In all cases, the conditions of \Cref{thm:main_theorem} are satisfied, and the hypergraph $H$ is obtained by removing $E\ev{L_r^{A}} \odot \{A\}$ and adding $E\ev{t\ev{ L_r^{A} }} \odot \{A\}$, for each $r$ and $A$, so the statement follows.
\end{proof}
\begin{example}
The $3$-uniform hypergraphs $G,H$ such that $V\ev{G} = V\ev{H} = \{ 0,\ldots, 6\} \cup \{x,y,z\}$ and edge sets given below satisfy $Q^\top \mathcal{A}_G Q = \mathcal{A}_H$, where $Q = R_{f} \oplus I_3$:

$ E\ev{G} = \{   124  ,  2 3 5 ,  3 4 6 ,  4 5 7 ,  5 6 1 ,  6 7 2 ,  7 1 3 ,  1 x y ,  2 x y ,  4 x y ,  x y z \} $

$ E\ev{H} = \{   3 5 6  ,  4 6 7 ,  5 7 1 ,  6 1 2 ,  7 2 3 ,  1 3 4 ,  2 4 5 ,  3 x y ,  5 x y ,  6 x y ,  x y z \}.$
\end{example}

\subsection{Cube hypergraph switching}
Recall that \cite[Theorem 31, p.~19]{abiad_berg_simoens} describes cube switching for graphs of rank $2$, using the sporadic indecomposable regular orthogonal matrix $R_{c}$ (q.v.~\Cref{def:all_switching_matrices}). We provide an example of cube switching using hypergraphs of rank $4$. 
\begin{example}
Consider $4$-uniform hypergraphs $G,H$ such that $V\ev{G} = V\ev{H} = \{ 1,\ldots, 8\} \cup \{ x, y, z\}$ and $Q = R_c \oplus I_3$ and $Q^\top \mathcal{A}_G Q = \mathcal{A}_H$.

$ E\ev{G} = (\{ 2,3,6,7 \} \odot \{ x, y, z\}) \cup (\{ 17, 26, 46, 48, 28, 35, 45, 47, 56, 67 \} \odot \{ x, y \})$

$ E\ev{H} = (\{ 1,3,6,8 \} \odot \{ x, y, z\}) \cup (\{ 17, 26, 46, 48, 12, 18, 24, 27, 36, 37 \} \odot \{ x, y \}).$ 
\end{example}

\section{Regularity of adjacency tensors of hypergraphs}\label{sec:regular_implies_complete}

In this section, we consider alternative ways to define the E-characteristic polynomial. 

Consider the equations $\mathcal{A}x-\lambda x = 0$ and $x^\top x = 1$ (\ref{eqn:e_char_res}). Given a normalized E-eigenpair $\ev{\lambda,x}$, it follows that $\ev{\ev{-1}^{k}\lambda,-x}$ is also a normalized E-eigenpair. If $k$ is odd, this implies that negating an $E$-eigenvalue $\lambda$ yields a distinct normalized $E$-eigenvalue, unless $\lambda$ is zero. 

To define the E-characteristic polynomial, we take a generic tensor $\mathbb{A}^n_k = \ev{ \mathbf{a}\ev{j_1,\ldots,j_k} }$ and define the ideal $J$ of $\C[\lambda, \{ \mathbf{a}\ev{j_1,\ldots,j_k} \}, x_1, \ldots, x_n]$ generated by the $n+1$ many polynomials $ (\mathbb{A}x - \lambda x)_{1\leq i \leq n}, x^\top x - 1$. Then the \textit{generic E-characteristic polynomial} is the unique monic polynomial $\phi\ev{\lambda}$ in $\C[\lambda][\mathbf{a}\ev{j_1,\ldots,j_k}]$ such that the elimination ideal $J \cap \C[\lambda][\mathbf{a}\ev{j_1,\ldots,j_k}]$ that results after eliminating $x_1,\ldots,x_n$ is generated by $\phi\ev{\lambda}$ if $k$ is even, and by $\phi\ev{\lambda^2}$ if $k$ is odd. As shown in \cite[Corollary 3.1, p.~946]{CARTWRIGHT2013942}, the generic E-characteristic polynomial is an irreducible polynomial in the ring $\C[\lambda][\mathbf{a}\ev{j_1,\ldots,j_k}]$, of degree $\sum_{i=0}^{n-1}\ev{k-1}^{i} = (\ev{k-1}^{n}-1)/(k-2)$ for $k\geq 3$ and of degree $n$ for $k=2$. See \cite{github} for a Sagemath code calculating E-characteristic polynomials. 

For a tensor $\mathcal{A}$ with complex number entries, we define the E-characteristic polynomial $\phi_\mathcal{A}\ev{\lambda}$ as the polynomial obtained by making substitutions $ \mathbf{a}\ev{j_1,\ldots,j_k} = \mathcal{A}_{j_1,\ldots,j_k}$ in the generic E-characteristic polynomial $\phi\ev{\lambda}$. (Since elimination and substitution do not commute, eliminating $x_1,\ldots,x_n$ from the ideal generated by $(\mathcal{A}x - \lambda x)_{1\leq i \leq n}$ and $x^\top x - 1$ does not necessarily yield the same polynomial.) If $k=2$, then $\phi_{\mathcal{A}}(\lambda)$ is just the classical characteristic polynomial of the square matrix $\mathcal{A}$. 

There is an alternative approach for defining the E-characteristic polynomial of the $E$-eigenvalue equations in \cite{Q2007}, where the resultant of the homogenized system is proposed:
\begin{equation} 
\phi'_{\mathcal{A}}(\lambda) :=\begin{cases}
\operatorname{Res}_x ( \mathcal{A}x - \lambda(x^\top x)^{\frac{k-2}2} x), & k \text{ is even}, \\
\operatorname{Res}_{x,\beta} \left(\begin{smallmatrix}
\mathcal{A}x-\lambda\beta^{k-2}x \\
x^\top x -\beta^2
\end{smallmatrix}\right), & k \text{ is odd}, \\
\end{cases}
\end{equation}
A tensor $\mathcal{A}$ is \textit{irregular} if there is a non-zero vector $x$ such that $\mathcal{A}x = 0$ and $x^\top x = 0$, and is \textit{regular}, otherwise. (``Regular matrix'' has a different meaning in \Cref{sec:indecomp}, but the context makes it clear which meaning is intended.) In \cite[Theorem 4]{Q2007}, it was shown that the normalized $E$-eigenvalues of a regular tensor are exactly the roots of $\phi'_{\mathcal{A}}(\lambda)$. In other words, the resultant of the homogenized system detects normalized $E$-eigenvalues for a regular tensor. On the other hand, the resultant is zero if the tensor is irregular and its rank is at least $3$. In \Cref{thm:qi-regular} below, we present a complete characterization of the \textit{regular} hypergraphs $G$ of rank $k\geq 2$, i.e., those hypergraphs with a regular adjacency tensor $\mathcal{A}_G$. It turns out that a hypergraph of rank $k\geq 3$ is likely irregular, whereas a graph (of rank $2$) is likely regular, by \cite[Theorem 1.2, p.~270]{costello_vu} and the proposition below.

\begin{remark}
\label{rmk:induced_subgraph_zero}
Let $H$ be a $3$-uniform hypergraph and $G$ be an induced subgraph of $H$ with $|V\ev{G}| = m \leq n = |V\ev{H}|$. If $G$ is irregular, then $H$ is also irregular. 
\end{remark}
\begin{proof}
If $G$ is irregular, then there is some $x\in \C^m \setminus \{\mymathbb{0}_m \}$ such that $\mathcal{A}_G x = \mymathbb{0}_m$ and $x^\top x = \mymathbb{0}_m$. Then, the vector $x \oplus \mymathbb{0}_{n-m} \in \C^n \setminus {\mymathbb{0}_n} $ satisfies the same equations for $H$. 
\end{proof}

Define the hypergraphs
\begin{align*} 
& G_1 = \ev{ [4], \{ 123, 234\} } && \text{ (3-uniform diamond hypergraph) }\\
& G_2 = \ev{ [4], \{ 123, 234, 124 \} } && \text{ (3-uniform complete hypergraph minus an edge) }\\
& G_3 = \ev{ [5], \{ 123, 345 \} } && \text{ (3-uniform loose path of length 2) }
\end{align*}

\begin{lemma}
\label{lem:special_examples}
The hypergraphs $G_1,G_2,G_3$ are irregular. 
\end{lemma}
\begin{proof}
The nonzero vectors below satisfy $\mathcal{A}_G x = x^\top x = 0$:

For $G_1$, consider $x = [1 , 0 , 0 , i , 0 ]^\top$. Then,
\begin{equation*}
\label{eqn:homogenous_G1}
\mathcal{A}_{G_1} x = \begin{bmatrix}
x_2 x_3  \\
x_1 x_3 + x_3 x_4 \\
x_1 x_2 + x_2 x_4  \\
x_2 x_3 
\end{bmatrix} = \mymathbb{0}_4
\end{equation*}

For $G_2$, define $x = [(-1/2)(1+i\sqrt{3}) , 0 , 1 , (-1/2)(1-i\sqrt{3}) , 0 ]^\top $. Then,
\begin{equation*}
\label{eqn:homogenous_G2}
\mathcal{A}_{G_2} x = \begin{bmatrix}
x_2 x_3 + x_2 x_4  \\
x_1 x_3 + x_3 x_4 + x_1 x_4   \\
x_1 x_2 + x_2 x_4  \\
x_2 x_3 + x_1 x_2  
\end{bmatrix} = \mymathbb{0}_4
\end{equation*}

For $G_3$, let $x = [1 , 1 , 0 , i , i, 0]^\top $.
\begin{equation*}
\label{eqn:homogenous_G3}
\mathcal{A}_{G_3} x = \begin{bmatrix}
x_2 x_3  \\
x_1 x_3  \\
x_1 x_2 + x_4 x_5   \\
x_3 x_5   \\
x_3 x_4  
\end{bmatrix} = \mymathbb{0}_5
\end{equation*} 
\end{proof}

\begin{lemma}
\label{lem:connected_implies_complete}
If a connected $3$-uniform hypergraph $H$ does not contain any induced subgraph isomorphic to $G_i$, $i=1,2,3$, then it is complete. 
\end{lemma}
\begin{proof}
Suppose the $3$-uniform hypergraph $H$ is connected, not complete, and does not contain an induced copy of $G_i$, $i=1,2,3$.  Let $\Gamma$ be the $2$-shadow of $H$, i.e., $V(\Gamma)=V(H)$ and $E(\Gamma) = \{xy : \exists e \in E(H) \textrm{ with } x,y \in e\}$.  If $\Gamma$ is not complete, then there is some pair $x,y \in V(\Gamma)$ so that the $x-y$ distance in $\Gamma$ is $2$.  Thus, there is a vertex $z$ so that $xz, yz \in E(\Gamma)$ and so there exists a vertices $w_1,w_2$ so that $xw_1z, yw_2z \in E(H)$.  If $\{x,y,z,w_1,w_2\}$ induces only two edges, then it is a $G_3$, a contradiction.  Thus, there is another edge $e$ of $H$ induced by this set.  If $e=w_1w_2z$ or $e=xw_1w_2$, then $\{x,w_1,w_2,z\}$ induces a $K^{(3)}_4$, so we may take $e=xzw_2$; similarly, if $e=yw_1w_2$, we may take $e=yzw_1$.  Thus, we may assume $e=tw_jz$ for $j \in \{1,2\}$ and $t \in \{x,y\}$.  But then $\{x,w_j,z,y\}$ induces a $G_1$ or $G_2$, a contradiction.  So, $w_1=w_2$, but then $\{x,y,z,w_1\}$ induces a $G_1$ or $G_2$, a contradiction. 

If, on the other hand, $\Gamma$ is complete, let $xyz$ be any non-edge of $H$.  Then there exist $w_1$, $w_2$ with $xw_1z,zw_2y \in E(H)$, and the above argument applies to this situation as well.
\end{proof}

Before the main result of this section, note that the ``complete'' $3$-uniform hypergraph on $2$ vertices is disconnected, so the size of a connected and complete $3$-uniform hypergraph is in $\{ 1 \}\cup \{ 3, 4, \ldots \}$. 
\begin{lemma}
\label{lem:complete_implies_regular}
Let $G$ be a $3$-uniform connected and complete hypergraph of size $\geq 3$. If $\mathcal{A}_G x = 0$, then $x=0$. In particular, $G$ is regular. 
\end{lemma}
\begin{proof}
We apply induction on the number of vertices, $|V\ev{G}| = n\geq 3$. If $n=3$, then by (\ref{eqn:tensor_times_vect}), we have $\mathcal{A}_G x = [2! x_2x_3 , 2! x_1x_3, 2! x_1x_2]^\top $. So, if $\mathcal{A}_G x = 0$ and $x^\top x = 0$, then it follows that $x=0$. For the inductive step, fix $n\geq 4$. Assume, for a contradiction, that $x$ is a nonzero vector such that $\mathcal{A}_G x = x^\top x = 0$. If $x$ has a zero entry, then deleting that vertex yields a connected and complete irregular subgraph on $n-1$ vertices, contradicting the inductive hypothesis. So, every entry of $x$ is non-zero. Then, for each distinct pair $i,j$, we obtain 
$$
0 = x_i \ev{ \mathcal{A}_G x }_i - x_j \ev{ \mathcal{A}_G x }_j = \ev{ x_i - x_j } \sum_{i_3 \in [n] \setminus \{ i,j\} } 2! x_{i_3}  
$$
which implies $x_1=\ldots = x_n$. Combining with $x^\top x = 0$, we obtain $x=0$. 
\end{proof}

\begin{theorem} \label{thm:qi-regular}
Given $k\geq 2$ and a $k$-uniform hypergraph $H$, then $H$ is regular if and only if:
\begin{enumerate}
\item[\textit{i}] $H$ is graph of rank $2$ and has nullity $\leq 1$. 
\item[\textit{ii}] $H$ is a disjoint union of $3$-uniform connected and complete hypergraphs, with at most one isolated vertex. 
\end{enumerate} 
\end{theorem}
\begin{proof}
If $k \geq 4$ then $x = [1,i,0,\ldots,0]^\top$ witnesses the irregularity of $H$. Hence, we only need to consider $k\in \{2,3\}$.

First, assume that $k=2$. As $\mathcal{A}_H$ is real and symmetric matrix, its null-space has a real orthonormal basis, $\{y_1,\ldots,y_\ell\}$, where $\ell \geq 0$ is the nullity. If the nullity is zero, then there is no non-zero $x$ such that $\mathcal{A}_H x = 0$, therefore $H$ is regular. Next, assume that the nullity is $1$. For any null-vector $x$, there is some $a\in \C \setminus \{0\}$ such that $x = a\cdot y_1$. It follows that $x^\top x = a^2 y_1^\top y_1 = a^2 \neq 0 $, so $H$ must be regular. Finally, assume that $\ell \geq 2$. Consider the complex vector $x := y_1 + i y_2$. It follows that $\mathcal{A}_H x = x^\top x = 0$, as $y_1$ and $y_2$ are orthonormal. As $y_1$ and $y_2$ are non-zero and have real coordinates, we also have $x \neq 0$, which proves that $H$ is irregular.

Next, consider $k = 3$. Assume that $H$ is regular. By \Cref{lem:special_examples} and \Cref{rmk:induced_subgraph_zero} Part 2, $H$ does not contain any subgraph isomorphic to $G_i$, for $i=1,2,3$. Then, by \Cref{lem:connected_implies_complete}, we infer that each connected component of $H$ is complete. If $H$ has more than one isolated vertices, say $v_1,v_2$, then $x = [1,i,0,\ldots,0]^\top$ shows irregularity, as in the case $k\geq 4$.

Conversely, assume that $H = H_1 \cup \ldots \cup H_m$ is a disjoint union of $3$-uniform connected complete hypergraphs with at most one isolated vertex. Let $x$ be a vector such that $\mathcal{A}_H x = x^\top x = 0$. Let $y_i$ be the vector where the entries of $y_i$ agree with those of $x$ corresponding to vertices of $H_i$, and so $x = y_1 \oplus \ldots \oplus y_m$ and $\mathcal{A}_{H_i} y_i = 0$. 

\textbf{Case 1:} If $H$ has no isolated vertex, then it follows by \Cref{lem:complete_implies_regular} that $y_i=0$, for each $i$, which implies $x=0$. 

\textbf{Case 2:} If $H$ has an isolated vertex, say $v_1$, then $y_i = 0$ for each $i = 2,3,\ldots, m$, as in Case 1. Then, we have $0 = x^\top x = y_1^2$, which implies $y_1=0$ and so, $x=0$. 
\end{proof}

\section{Future Directions}

Here, we mention a few open problems which arose in the context of the present discussion.

\begin{enumerate}
\item It was noted in \Cref{prop:exhaustive_data} that the equation $Q^\top \mathcal{A}_G Q = \mathcal{A}_H$ is satisfied for $Q = R_f$, by the $3$-uniform Fano planes $G$ and $H$. Are there non-isomorphic hypergraphs $G, H$ of rank $k\geq 3$ and size $|Q|$ (defined only on the ``switching set'') that satisfy this equation for one of the indecomposable regular orthogonal matrices $Q$ in \Cref{def:gm_and_wqh_matrices} or \Cref{def:all_switching_matrices}?
\item Recall that $\mathcal{I}^k_n$ denotes the identity tensor of order $k$ and dimension $n$. It is known that (\cite[Theorem 2.1, p.~2355]{shao}) if $Q^\top \mathcal{A}_G Q = \mathcal{A}_H$ and $Q^\top \mathcal{I}^k_n Q = \mathcal{I}^k_n$ for some square matrix $Q$, then the uniform hypergraphs $G$ and $H$ are cospectral in the sense of \cite{C2012} (and also $E$-cospectral). The converse is true for graphs, i.e., $k=2$, but is it true for $k > 2$?  More generally, can one use switching to obtain cospectral hypergraph pairs?
\end{enumerate}


\subsection*{Acknowledgements}
\raggedright
Aida Abiad is supported by the Dutch Research Council (NWO) through the grant VI.Vidi.213.085. 

\printbibliography 
\end{document}